\documentclass[11pt, a4paper]{amsart}
\usepackage[left=1in, right=1in]{geometry}
\usepackage{bbm}
\usepackage{float, graphicx}
\usepackage[]{epsfig}
\usepackage{amsmath, amsthm, amssymb, mathscinet}
\allowdisplaybreaks
\usepackage{epsfig}
\usepackage{verbatim}
\usepackage{multicol}
\usepackage{url}
\usepackage{latexsym}
\usepackage{mathrsfs}
\usepackage[colorlinks, bookmarks=true]{hyperref}
\usepackage{graphicx}
\usepackage{enumerate}
\usepackage[normalem]{ulem}
\usepackage{bm}
\usepackage{dsfont}
\usepackage{stmaryrd}
\usepackage{enumitem}
\usepackage{soul}
\usepackage{color}
\usepackage{xcolor}
\usepackage{hyperref}
\usepackage{indentfirst}
\usepackage[capitalise]{cleveref}

\numberwithin{equation}{section}
\usepackage{mathtools}

\newcommand{\bN}{{\mathbb{N}}}
\newcommand{\bR}{{\mathbb R}}
\newcommand{\bS}{{\mathbb S}}

\newcommand{\bE}{{\mathbb E}}

\DeclareMathOperator{\Airy}{Ai}

\def\<{\langle}
\def\>{\rangle}

\def\Lamp[#1]{\boldsymbol{\Lambda}_{\mathrm{AMP}}^{(#1)}}
\def\lalg[#1]{\Lambda_{\mathrm{alg}, #1}}

\def\de{{\rm d}}

\def\RR{\mathbb{R}}
\def\SS{\mathbb{S}}

\def\calI{\mathcal{I}}
\def\calJ{\mathcal{J}}

\def\calW{\mathcal{W}}

\newcommand{\R}{\mathbb{R}}

\newcommand{\eps}{\varepsilon}

\newcommand{\dist}{\operatorname{dist}}

\newcommand{\norm}[1]{\left\lVert#1\right\rVert}
\newcommand{\RN}[1]{%
  \textup{\uppercase\expandafter{\romannumeral#1}}%
}

\newcommand{\RNum}[1]{\uppercase\expandafter{\romannumeral #1\relax}}

\theoremstyle{plain} 
\newtheorem{theorem}{Theorem}[section]
\newtheorem*{theorem*}{Theorem}
\newtheorem{lemma}[theorem]{Lemma}
\newtheorem*{lemma*}{Lemma}

\newtheorem*{corollary*}{Corollary}
\newtheorem{proposition}[theorem]{Proposition}
\newtheorem*{proposition*}{Proposition}

\newtheorem*{assumption*}{Assumption}

\newtheorem*{definition*}{Definition}

\newtheorem*{example*}{Example}
\newtheorem{remark}[theorem]{Remark}

\newtheorem*{remark*}{Remark}
\newtheorem*{remarks*}{Remarks}

\title[clamped plate resembles membrane]{A vibrating clamped plate resembles a vibrating membrane at high frequency}
\date{}

\author{Or Kuperman}
\address{(OK) Einstein Institute of Mathematics, Edmond J. Safra Campus, The Hebrew University of Jerusalem, Jerusalem 9190401, Israel} 
\email{or.kuperman@mail.huji.ac.il}

\author{Zhengjiang Lin}
\address{(ZL) Department of Mathematics, Massachusetts Institute of Technology, 77 Massachusetts Ave, 02139 Cambridge MA, USA} 
\email{linzj@mit.edu}

\author{Dan Mangoubi}
\address{(DM) Einstein Institute of Mathematics, Edmond J. Safra Campus, The Hebrew University of Jerusalem, Jerusalem 9190401, Israel} 
\email{dan.mangoubi@mail.huji.ac.il}

\begin{document}

\begin{abstract}
    We study the clamped plate eigenvalue problem. We show that clamped plate eigenfunctions can be approximated in the $L^2$-sense by their oscillatory components, which solve the Helmholtz equation, with an error that decays exponentially fast in the frequency on any compact subdomain. Additionally, we establish an upper bound on the boundary localization of clamped plate eigenfunctions analogous to the one observed for membrane eigenfunctions.
\end{abstract}

\maketitle

\section{Introduction}
We consider the eigenvalue problem for a clamped plate under tension in a bounded smooth domain $\Omega\subseteq\bR^d$. For clarity of presentation, we first describe the case without tension:
\begin{align}\label{eq:clamped-plate}
\begin{cases}
    \Delta^2 u = \lambda u & \text{in } \Omega, \\
    u = \partial_\nu u = 0 & \text{on } \partial\Omega,
\end{cases}
\end{align}
where $\partial_\nu$ is the outward unit normal. A classical motivation for studying \eqref{eq:clamped-plate} comes from the celebrated experiments of Chladni \cite{chladni}. When a metal plate sprinkled with fine particles is set into vibration, the particles migrate toward the regions of zero displacement and form striking patterns, known as Chladni figures (see \Cref{fig:chladni 1}). These patterns naturally motivate the study of the nodal set of a plate eigenfunction. In the membrane case, namely for Laplacian eigenfunctions, the structure of nodal sets has been extensively studied in recent years; see, for instance, \cite{bruning1978, yau-problem1982, donnelly-fefferman1988, logunov-malinnikova-2-3,logunov-lower, logunov-upper,MR4514975}. By contrast, much less is known about the nodal sets of clamped plates, even though the clamped plate problem is the classical mathematical model underlying Chladni's experiments. Indeed, in Kirchhoff--Love thin-plate theory, bending stiffness leads to a fourth-order equation rather than a second-order one \cite{L88,L44,rayleigh}. After separation of variables, the transverse vibration is governed by a bi-Laplacian eigenvalue problem such as~\eqref{eq:clamped-plate}. Nevertheless, simulated Laplacian nodal patterns and experimental Chladni figures can be remarkably similar (cp.\ \Cref{fig:chladni 1} and \Cref{fig:laplacian 1}). Although the present paper does not directly address nodal sets, its purpose is to provide a quantitative explanation for the similarity of these two models.

To begin with, observe that on a closed manifold the maximum principle, or alternatively the spectral theorem, shows that the bi-Laplacian eigenfunctions are the same as Laplacian eigenfunctions. In other words, in the absence of a boundary the plate problem coincides with the membrane problem. In the presence of a boundary, however, the clamped conditions $u = \partial_\nu u = 0$ impose two boundary constraints and prevent such an immediate identification. The main point of the present paper is that, nonetheless, the high-frequency clamped-plate problem remains extremely close to a second-order membrane problem.


\begin{figure}[ht]
\centering
\includegraphics[width=0.4\textwidth]{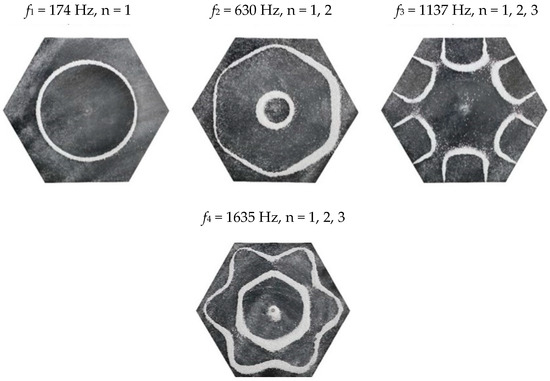}
    \caption{Chladni figures: nodal line patterns observed in a $14.5$ cm hexagon plate (Figure 11 in \cite{BC24}).}    \label{fig:chladni 1}
\end{figure}


More precisely, any clamped plate eigenfunction $u_{\lambda}$ of eigenvalue~$\lambda$ admits a decomposition into an oscillating component $u_{\lambda, o}$, which solves the Helmholtz equation with a positive eigenvalue, and an exponential component $u_{\lambda, e}$, which solves the Helmholtz equation with a negative eigenvalue. We write
\begin{align*}
u_{\lambda} = u_{\lambda, o}+u_{\lambda, e} ,
\end{align*}
where $u_{\lambda, o} = \frac{1}{2}\bigl(u_{\lambda} - \lambda^{-1/2} \Delta u_{\lambda}\bigr)$ and $u_{\lambda, e} = \frac{1}{2}\bigl(u_{\lambda} + \lambda^{-1/2} \Delta u_{\lambda}\bigr)$
yielding
\begin{equation*}
-\Delta u_{\lambda, o} =\lambda^{1/2} u_{\lambda, o} \qquad \text{ and } \qquad
-\Delta u_{\lambda, e} =-\lambda^{1/2} u_{\lambda, e
}.
\end{equation*}
As mentioned above, on closed manifolds $u_{\lambda, e}=0$. In this paper, we prove that $u_{\lambda, e}$ remains small in an $L^2$-sense even in the presence of a boundary. Orthogonality (see Lemma~\ref{lem:orthogonal u v}) shows that
\[\|u_\lambda\|^2_{L^2(\Omega)}=\|u_{\lambda, e}\|^2_{L^2(\Omega)}+\|u_{\lambda, o}\|^2_{L^2(\Omega)}\ .\]
Our first main result gives a quantitatively decaying bound on the exponential component.
\begin{theorem}
    Let $\Omega\subseteq \R^d$ be a bounded $C^4$-domain. Let $u_\lambda$ be any clamped plate eigenfunction of eigenvalue~$\lambda$. We have
        \begin{align*}
            \|u_{\lambda, e}\|^2_{L^2(\Omega)}\leq C \lambda^{-1/4}\|u_{\lambda}\|^2_{L^2(\Omega)}.
        \end{align*}
        where $C=C(\Omega)$ is a positive constant which depends on~$\Omega$ and is independent of~$\lambda$.
\end{theorem}
Indeed, we show that 
$\|u_{\lambda, e}\|_{L^2(\partial\Omega)}$
is bounded in terms of $\|u\|_{L^2(\Omega)}$ and 
 if we let 
 \begin{equation}
 \label{eq:omegas}
     \Omega_s=\{x\in\Omega|\dist(x, \partial\Omega)>s\}
\end{equation}
 then $\|u\|_{L^2(\partial\Omega_s)}$ and  $\|u\|_{L^2(\Omega_s)}$ decay exponentially with the frequency.
 \begin{theorem}
 \label{thm:approx-and-decay-intro}
  Let $\Omega$ be as above. For any $u_{\lambda}$ as above
    \begin{align*}
             \|u_{\lambda, e}\|^2_{L^2(\partial\Omega)} \leq C \|u_{\lambda}\|^2_{L^2(\Omega)}.
    \end{align*}
    Moreover,
     \begin{align*}
             \|u_{\lambda, e}\|^2_{L^2(\partial\Omega_s)} \leq Ce^{-\frac{1}{2}\lambda^{1/4} s} \|u_{\lambda}\|^2_{L^2(\Omega)},
    \end{align*} 
      and
      \[\|u_{\lambda, e}\|^2_{L^2(\Omega_s)}\leq C\lambda^{-1/4}e^{-\frac{1}{2}\lambda^{1/4}s}\|u_{\lambda}\|^2_{L^2(\Omega)},\]
         for $0\leq s<\delta$, where $\delta=\delta(\Omega)$ is small enough. 
\end{theorem}

In the next two theorems we show delocalization properties of the clamped plate eigenfunctions near the boundary. 
In the case of membranes one knows that $\lambda$-eigenfunctions cannot localize near the boundary at a scale of~$o(\lambda^{-1/3})$ as $
\lambda\to\infty$~(see e.g. \cite[Lemma 3.2]{hassell-tao-mrl} and \cite{hassell-tao-mrl-erratum}). It turns out that a similar delocalization holds also for clamped plate $\lambda$-eigenfunctions.
\begin{theorem}[Delocalization near the boundary]
\label{thm:delocalization intro}
     Let $\Omega$ be as above.  Let $u_\lambda$ be as above. For any $\gamma \in (0,1)$, let $s(\lambda) = \gamma \lambda^{-1/6}$. Then
 \begin{align*}
            \|u_{\lambda}\|^2_{L^2(\Omega\setminus\Omega_{s(\lambda)})}\leq C \gamma  \|u_{\lambda}\|^2_{L^2(\Omega)}
        \end{align*}
        for some $C=C(\Omega)>0$.
\end{theorem}
The delocalization scale $\lambda^{-1/6}$ in \Cref{thm:delocalization intro} is optimal when $\Omega = B_1\subseteq \bR^d$ is the unit ball for $d \geq 2$. In other words, there exist solutions $u_{\lambda}$ to \eqref{eq:clamped-plate} that concentrate near the boundary precisely at this scale $\lambda^{-1/6}$.
\begin{theorem}[Boundary localization in balls]\label{thm:whispering gallery} Consider the clamped plate problem in the unit ball~$\Omega = B_1$ of~$\RR^d$ for $d \geq 2$. There exists a sequence of $L^2$-normalized eigenfunctions $(u_{m})_{m=1}^{\infty}$ with the following property: if we denote the corresponding eigenvalues by~$(\lambda_{m})_{m=1}^{\infty}$, then   for all $\eps>0$ there exists  $\gamma(\eps)>0$, such that for $s(m)=\gamma(\eps)\lambda_m^{-1/6}$,
    \begin{align*}
          \lim_{m \to \infty} \|u_{m}\|^2_{L^2(\Omega \setminus \Omega_{s(m)})} > (1-\eps).
     \end{align*}
\end{theorem}

We finally show that one has no localization on the flat cylinder. We prove
\begin{theorem}[No localization in flat cylinders]\label{thm:delocalization on cylinder}
Consider the clamped plate problem in the cylinder $\Omega = \bS^1 \times [-1, 1] $.  Let $(u_\lambda)_{\lambda}$ be any sequence of $L^2$-normalized eigenfunctions with $\lambda\to\infty$. For any $s \in (0,1)$, there exists $\delta = \delta(s) >0$ such that
       \begin{align*}
            \varlimsup_{\lambda\to\infty}\|u_{\lambda}\|^2_{L^2(\Omega\setminus\Omega_s)} < (1-\delta) \quad  \mathrm{ and }\quad
            \varlimsup_{\lambda\to\infty}\|u_{\lambda}\|^2_{L^2(\Omega_s)} < (1-\delta).
       \end{align*}
       Hence, eigenfunctions neither localize near the boundary nor localize in the interior of the domain.
\end{theorem}

\begin{figure}[ht]
\centering
    \includegraphics[width=0.4\textwidth]{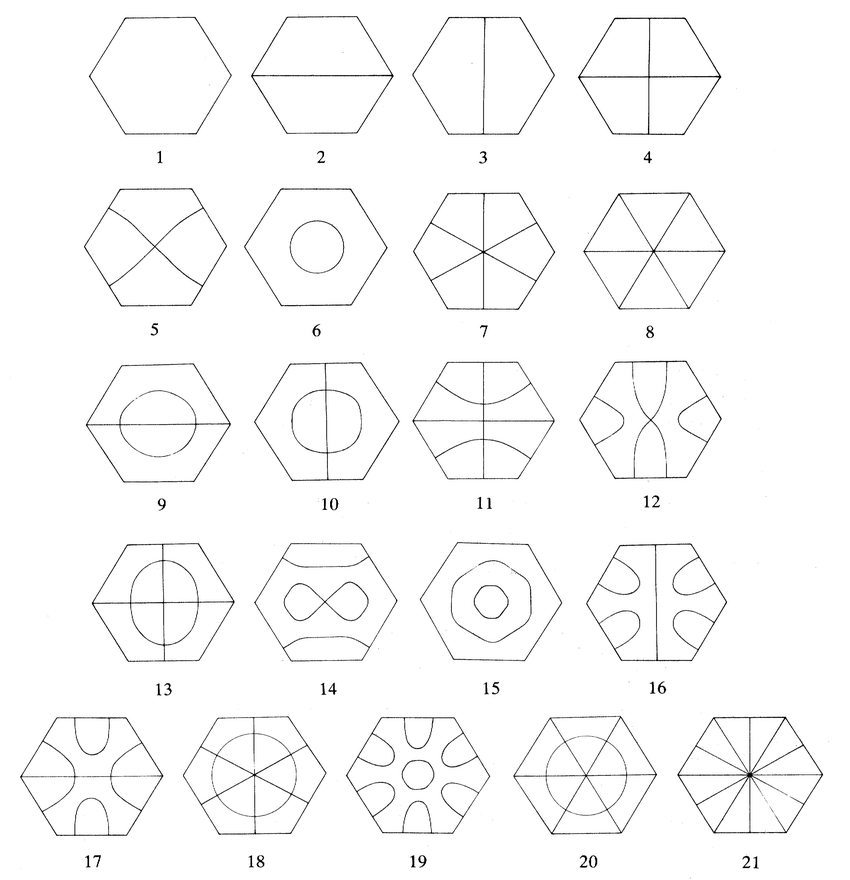}
    \caption{Nodal sets of Dirichlet Laplacian eigenfunctions in a hexagon (Figure 2 in \cite{BR78}).}
    \label{fig:laplacian 1}
\end{figure}

\subsection{Motivation and related open directions}
Our main motivation for this work was a natural question posed to us by Iosif Polterovich, which remains open: does the nodal set of a clamped plate eigenfunction become dense at the wavelength scale $\lambda^{-1/4}$ as $\lambda\to\infty$? It is tempting to compare the plate eigenfunction with its oscillating component, whose nodal set is dense at the corresponding wavelength scale \cite{bruning1978, yau-problem1982}. However, Enciso and Greilhuber~\cite{enciso2025non} recently showed, surprisingly, that no universal density holds: there is no constant $C>0$ such that for every smooth domain~$\Omega\subset\bR^2$ and every clamped plate eigenfunction~$u_\lambda$, the nodal set $\{u_\lambda=0\}$ is $C\lambda^{-1/4}$-dense.
Nevertheless, the question whether in a given domain~$\Omega$ the nodal sets become dense remains open. 

One may also ask whether a density result  holds under  generic or curvature assumptions. For example, if one could show that the oscillating component~$u_{\lambda,o}$ has sufficiently large $\sup$-norm, say, at least inverse polynomial in the frequency, on every ball of the wavelength scale, then such a density statement would follow from the approximation result in \Cref{thm:approx-and-decay-intro}. This leads naturally to several questions concerning the small-scale delocalization of eigenfunctions.
For membrane eigenfunctions~$v_\lambda$ on a manifold~$M$, the pointwise Weyl law \cite{hormander-pointwise-weyl} implies that the sequence  $|v_\lambda(x)|^2$ is Ces\`aro summable uniformly in~$x$ to the constant~$1/\mathrm{Vol}(M)$ as $\lambda \to \infty$. In particular, the Ces\`aro means are uniformly bounded from below. It is natural to ask whether stronger forms of convergence hold under generic assumptions or geometric conditions, or after removing a density-zero subsequence of eigenfunctions. One may similarly ask whether the $L^1$-norm on every ball of the wavelength scale admits a uniform lower bound, which would provide a quantitative form of delocalization in the spirit of Steinerberger's question \cite{steinerberger-L1}. It is plausible that delocalization results of this type for membrane eigenfunctions could also be adapted to the oscillating components of clamped plate eigenfunctions. Relevant global results include the Shnirelman--Zelditch--Colin de Verdi\`ere quantum ergodicity theorem \cite{shnirelman1973, zelditch1987, CdV1985}, which states, say on closed manifolds, that when the geodesic flow is ergodic, the probability measures $|v_\lambda(x)|^2\, \de x$ converge to the uniform measure along a density-one subsequence of eigenfunctions. We also mention the theorem of Dyatlov--Jin--Nonnenmacher \cite{Dyatlov-Jin2018, Dyatlov-Jin-Nonnenmacher2022}, which gives on closed manifolds of negative curvature a uniform lower bound for $\|v_\lambda\|_{L^2(U)}$ on every fixed open set~$U$. These results, however, concern fixed spatial scales and do not capture the behavior of eigenfunctions on the wavelength scale.

Finally, we point to a closely related direction for future investigation. The boundary conditions for a freely vibrating plate, or more generally a shell, are considerably more intricate; see, for instance, \cite[ch. II]{landau-lifshitz-elasticity}. In particular, their description requires more refined mechanical models than the clamped plate problem considered here; we refer to the discussions in \cite[ch.\ X, \S224]{rayleigh}, \cite[ch.\ 4, \S22]{timoshenko-plates1959}, and \cite[ch.\ 3.5, ch.\ 10]{reddy}. It would be interesting to determine whether analogous reductions to second-order equations remain valid in these settings and, if so, how such reductions affect localization and delocalization phenomena as well as the resulting nodal patterns.

\clearpage\section*{Acknowledgments} 
We are very grateful to Sasha Logunov for his guidance, accompanying this project, several valuable  questions, insightful discussions and encouragement, and  to David Jerison for his invaluable discussions and for suggesting the title of this paper.
We are grateful to Iosif Polterovich who raised the question on the density of zeros for the clamped plate, serving as the initial motivation for this work. We are also grateful to Josef Greilhuber for illuminating discussions.
DM would like to thank the hospitality of the Mathematics department at MIT where part of this work was done.

\section{The clamped plate under tension and preliminary estimates}\label{sec:preliminary}

We start by generalizing the clamped plate eigenvalue problem \eqref{eq:clamped-plate}. In more practical plate models as in \cite{WC43,L44,rayleigh,GHS10} the equation in~\eqref{eq:clamped-plate} usually includes  an additional term~$\tau \Delta u$, where the positive tension parameter $\tau$ depends on the thickness, densities, Young's modulus, and Poisson's ratio of the metal plate. Hence, it leads to the boundary problem
    \begin{align}\label{eq:main clamped}
        \begin{cases}
            Lu := \Delta^2 u - \tau \Delta u - \lambda u = 0 \text{ in } \Omega,
            \\
            u = \partial_{\nu} u = 0 \text{ on } \partial\Omega.
        \end{cases}
    \end{align}
    Similarly to the case where $\tau=0$ we can define oscillating and exponential parts. Set $-\alpha, \beta$ to be the roots of the equation $x^2-\tau x-\lambda =0$.
    \begin{align*}
        \alpha \coloneqq \left(\frac{\tau ^2}{4} + \lambda\right)^{1/2}- \frac{\tau}{2}, \quad\beta \coloneqq  \left(\frac{\tau ^2}{4} + \lambda\right)^{1/2}+\frac{\tau}{2},
    \end{align*}
and let 
    \begin{align*}
        u_o \coloneqq \frac{1}{\alpha + \beta} (\beta-\Delta)u,
            \quad
            u_e \coloneqq \frac{1}{\alpha + \beta} (\alpha+\Delta)u.
    \end{align*}
    
\begin{lemma}[see \cite{WC43}]
\label{lem:two eigenfunctions}
 We have $u=u_o+u_e$, $L=(\Delta+\alpha)(\Delta-\beta)$ and
    \begin{align*}
        \begin{cases}
            (\Delta + \alpha)u_o = 0 \text{ in } \Omega, \\
            (\Delta - \beta)u_e = 0 \text{ in } \Omega,\\
            -u_o=u_e = \frac{\Delta u}{\alpha + \beta} \text{ on }\partial\Omega,\\
            -\partial_{\nu}u_o= \partial_{\nu}u_e = \frac{\partial_{\nu} \Delta u}{\alpha + \beta} \text{ on } \partial\Omega.
        \end{cases}
    \end{align*}
\end{lemma}
\begin{proof}
Clear.
\end{proof}
We prove approximation and delocalization in the following two theorems.
\begin{theorem}[Approximation by Helmholtz solutions]\label{thm:approx-and-decay}
    Assume that $\Omega \subseteq \R^d$ is a~$C^4$-bounded domain.
    There are positive constants $C,\beta_0,\delta$ depending only on $\Omega$, such that for any solution~$u$ to the clamped plate eigenvalue problem \eqref{eq:main clamped} with $\beta \geq \beta_0$, for any $s \in[0,\delta]$,
        \begin{align}\label{ineq:main exp}
             \|u_e\|^2_{L^2(\partial\Omega_s)} \leq C e^{-s \sqrt{\beta}} \|u\|^2_{L^2(\Omega)}.
        \end{align}
        We also have
        \begin{align}\label{ineq:main global} 
              \|u_e\|^2_{L^2(\Omega_s)}\leq C \beta^{-1/2} e^{-s \sqrt{\beta}} \|u\|^2_{L^2(\Omega)}.
        \end{align}
        In particular, 
        \begin{align*}
              \|u_e\|^2_{L^2(\Omega)}\leq C \beta^{-1/2} \|u\|^2_{L^2(\Omega)}.
        \end{align*}
\end{theorem}
\begin{theorem}[Delocalization near the boundary]\label{thm:delocalization}
Let $u$ be a solution to Problem~\eqref{eq:main clamped}. The function $u_o$ delocalizes near the boundary in the sense that for any $s \in[0,\delta]$,
        \begin{align*}
            \|u_o\|^2_{L^2(\Omega\setminus\Omega_s)}\leq C(s + \beta s^3)\|u\|^2_{L^2(\Omega)}.
        \end{align*}
In particular, for any $s = \gamma \beta^{-1/3}$ with $\gamma \in (0,1)$,
    \begin{align*}
  \|u_o\|^2_{L^2(\Omega\setminus\Omega_s)}\leq C\gamma\|u\|^2_{L^2(\Omega)}.
        \end{align*}
\end{theorem}

 In the remainder of this section, we collect several basic estimates for solutions~$u$ to~\eqref{eq:main clamped}. 
\subsection{Basic definitions and facts for the clamped plate eigenvalue problem}
 We assume $ \beta > 1  $ in all estimates presented in \Cref{sec:preliminary}.
Recall the following orthogonality property (see \cite{WC43,L44,rayleigh,GHS10}).
\begin{lemma}\label{lem:orthogonal u v}
    The functions $u_o, u_e$ are orthogonal, i.e., 
    \begin{align*}
        \int_{\Omega} u_o u_e\,\de x= 0.
    \end{align*}
    Hence, 
    \[\|u\|^2_{L^2(\Omega)}=\|u_o\|^2_{L^2(\Omega)}+\|u_e\|^2_{L^2(\Omega)}\]
\end{lemma}
\begin{proof}
Integration by parts with the clamped boundary condition shows
    \begin{align*}
       (\alpha + \beta)^2 \int_{\Omega} u_o u_e\,\de x&= \int_{\Omega} \bigl((\beta-\Delta)u\bigr) \cdot \bigl((\alpha + \Delta)u\bigr) \,\de x\\&= \int_{\Omega} \bigl(( \alpha + \Delta)(\beta - \Delta)u\bigr) \cdot u \,\de x=-\int_{\Omega} Lu\cdot u\,\de x= 0.
   \end{align*}
\end{proof}

\subsection{\texorpdfstring{$H^4$}{H4}-estimates for \texorpdfstring{$u$}{u}}

We first recall that a solution~$u$ to~\eqref{eq:main clamped} belongs to  the Sobolev spaces $H^2_0(\Omega)$ and $H^4(\Omega)$ since~$\Omega$ is a~$C^4$-domain.

\begin{lemma}[$H^2$-estimates]\label{lem:w22}
    We have
    \begin{align*}
        \int_{\Omega} |\nabla ^2 u|^2  \,\de x= \int_{\Omega} (\Delta u)^2  \,\de x\leq 2 \beta ^2  \int_{\Omega}  u^2 \,\de x,
    \end{align*}
    and 
    \begin{align*}
        \int_{\Omega} |\nabla u|^2 \,\de x\leq 2 \beta \int_{\Omega}  u^2 \,\de x.
    \end{align*}
\end{lemma}
\begin{proof}
    The first standard identity holds since $u\in H^2_0(\Omega)$ (see for example Theorem 9.9 in \cite{GT77}). Integration by parts shows
    \begin{align*}
        \int_{\Omega} (\Delta u) ^2  \,\de x=  \int_{\Omega} (\Delta^2 u )u \,\de x= \int_{\Omega}\tau (\Delta u) u + \lambda u^2 \,\de x\leq \frac{1}{2}\int_{\Omega} (\Delta u)^2  \,\de x+ \frac{\tau^2+ 2 \lambda}{2} \int_{\Omega} u^2 \,\de x.
    \end{align*}
    Hence, $\int_{\Omega} (\Delta u) ^2  \,\de x\leq (\tau^2+2\lambda) \int_{\Omega}  u^2 \,\de x\leq 2\beta^2  \int_{\Omega}  u^2 \,\de x$. Also,
    \begin{align*}
        \int_{\Omega} |\nabla u|^2  \,\de x = \int_{\Omega} - \Delta u \cdot u \,\de x \leq \int_{\Omega} \frac{1}{4\beta} (\Delta u)^2 + \beta u^2 \,\de x \leq 2 \beta  \int_{\Omega}  u^2 \,\de x .
    \end{align*}
\end{proof}

\begin{lemma}[$H^4$-estimates]\label{lem:w42}
    There is a positive constant $C= C(\Omega)$ such that
        \begin{align*}
            \int_{\Omega} |\nabla^3 u|^2 \,\de x\leq C \beta^3 \int_{\Omega} u^2 \,\de x\quad \text{ and }\quad
            \int_{\Omega} |\nabla^4 u|^2  \,\de x \leq C \beta^4 \int_{\Omega} u^2 \,\de x.
       \end{align*}
\end{lemma}
\begin{proof}
    We first remark that we can obtain the third derivatives bound by combining the fourth derivatives bound and the Gagliardo-Nirenberg interpolation inequality on the domain $\Omega$ (see for example \cite[Th.~1.3]{LZ21},\cite{G59,N66,B19}). Rearrange the equation $\Delta^2 u - \tau \Delta u - \lambda u = 0$ to the form $\Delta^2 u = \tau \Delta u + \lambda u$, considering  \eqref{eq:main clamped} as a Dirichlet boundary value problem for~$\Delta^2$. The fourth derivatives bound then follows from the following standard global elliptic regularity for a Dirichlet boundary value problem in~$C^4$ domains (see, e.g. \cite[Th. 7.32 \& \S G]{F95} or \cite[Th. 2.20]{GHS10}) combined with Lemma~\ref{lem:w22}.
    \begin{align*}
         \int_{\Omega} |\nabla ^4 u|^2  \,\de x \leq C  \left( \int_{\Omega} \left(\tau \Delta u + \lambda u \right) ^2  \,\de x +\int_{\Omega} u^2 \,\de x \right).
    \end{align*}
\end{proof}

\subsection{Boundary \texorpdfstring{$L^2$}{L2}-Estimates for \texorpdfstring{$\Delta u$}{Laplace u}}
\label{subsec:boundary-Delta-estimates}
We estimate the boundary $L^2$-norm of $\Delta u$. One should compare the  method in this section to~\cite{hassell-tao-mrl} where boundary estimates for the normal derivative of a Dirichlet Laplace eigenfunction were obtained.

Recall the Fermi coordinates: Let $\Omega_s$ be as in~\eqref{eq:omegas}. There exists $\delta_0 = \delta_0(\Omega) \in (0,1)$ such that if we let $U=[0,\delta_0] \times \partial \Omega$ then the map $\Phi:U\to\overline{\Omega}\setminus\Omega_{\delta_0}$ defined by $(s,y)\mapsto y+sn(y)$ where $n(y)$ is the inward normal on~$\partial\Omega$ is a diffeomorphism. Observe that the horizontal section $s=s_0$ is mapped under~$\Phi$ to~$\partial\Omega_{s_0}$.
We can decompose the standard Euclidean metric in the form 
\begin{align*}
    \de x^2= \de s^2 + g_s,
\end{align*}
where $g_s$ is the induced metric on $\partial \Omega_s$. Define the second fundamental form of $\partial\Omega_s$ by the convention $\text{II}_s(X, Y)=-\langle \nabla^{\mathbb{E}}_X Y , \partial_s\rangle$, where $\nabla^{\mathbb{E}}$ is the standard connection on Euclidean space, and define the mean curvature~$H_s=\mathrm{tr}(\text{II}_s)$ with respect to the metric $g_s$ on $\partial\Omega_s$. Then, the Laplacian is written as
\begin{equation}
\label{eqn:Laplace in Fermi}
    \Delta = \partial_s^2 + H_s \partial_s + \Delta_{g_s}
\end{equation}
where $\Delta_{g_s}$ is the Laplace-Beltrami operator of~$\partial\Omega_s$ (see \cite{li-geometric-analysis}, \cite{colding-minicozzi-minimal-surfaces}).
A main step in the proof of Theorem~\ref{thm:approx-and-decay}
is given by the following.
\begin{proposition}\label{prop:boundary laplacian estimates}
There is a positive constant $C= C(\Omega)$ such that
    \begin{align*}
        \int_{\partial\Omega} u_e^2\,\de\sigma =\frac{1}{(\alpha+\beta)^2}\int_{\partial \Omega} (\Delta u)^2\,\de\sigma \leq C \int_{\Omega} u^2 \,\de x.
    \end{align*}
\end{proposition}

\begin{proof}
    Let $A(s)=\eta(s)\partial_s$ where $\eta(s)$ is  a cut-off function such that $\eta(s) = 1$ when $s\in[0,\delta_0/2]$ and $\eta(s) = 0$ when $s \geq \delta_0$, and $|\eta^{(k)}(s)| \leq C\delta_0^{-k}$ for $0\leq k\leq 4$.
  The vector field $A(s)$ extends $-\partial_{\nu}$ on $\partial\Omega$. 
  Observe that due to the boundary conditions and~\eqref{eqn:Laplace in Fermi}
  \[u|_{\partial\Omega}=(Au)|_{\partial\Omega}=0 \text{ and } (\Delta u)|_{\partial\Omega}=(\partial_s^2 u)|_{\partial\Omega}=\partial_s|_{s=0}(Au).\] Hence,  we can rewrite $(\Delta u)^2$ as $\Delta u \cdot \partial_s (Au)$, and we have by Green's theorem
    \begin{align}\label{eq:boundary Laplacian w}
      \begin{split}
          \int_{\partial \Omega} (\Delta u )^2\,\de\sigma  &= \int_{\partial \Omega} (\Delta - \beta)u \cdot \partial_s (Au) - \partial_s (\Delta u - \beta u )\cdot  A u\,\de\sigma 
   \\  &=  -\int_{ \Omega} (\Delta - \beta)u \cdot (\Delta + \alpha)(Au) - (\Delta + \alpha)(\Delta - \beta)u \cdot  Au\,\de x
     \\  &= -\int_{ \Omega} (\Delta - \beta)u \cdot (\Delta + \alpha)(Au)\,\de x= -\int_{ \Omega} u \cdot L(Au)\,\de x.
    \end{split}
   \end{align}
Since the commutator $[\Delta, A]$ is a second order operator and $[\Delta^2, A]$ is a fourth order one, both with no zero-order terms, we get
        \begin{align}\label{eq:commutator 3}
                    \|LAu\|^2_{L^2} &= \bigl\|[L, A]u\bigr\|^2_{L^2}  \leq 
    2\bigl\|[\Delta^2, A]u\bigr\|^2_{L^2}+2\tau^2\bigl\|[\Delta,A] u\bigr\|^2_{L^2}\\ 
    &\leq C\bigl(\|\nabla^4 u\|^2_{L^2}+\|\nabla^3 u\|^2_{L^2}\bigr)+C\tau^2\bigl(\|\nabla^2 u\|^2_{L^2}+\|\nabla u\|^2_{L^2}\bigr)\leq C\beta^4 \|u\|^2_{L^2}\nonumber
        \end{align}
where the last inequality follows from Lemmas~\ref{lem:w22} and~\ref{lem:w42}.
Finally, from \eqref{eq:boundary Laplacian w} and \eqref{eq:commutator 3} and the Cauchy-Schwarz inequality we see that
    \begin{align*}
    \int_{\partial \Omega} (\Delta u )^2\,\de\sigma \leq \| u\|_{L^2}\cdot\|LA u\|_{L^2}\leq C \beta^2\int_{\Omega} u^2\,\de x.
    \end{align*}
\end{proof}

\section{Proof of Theorem \ref{thm:approx-and-decay}: Approximation by Helmholtz solution}
\label{sec:pf-of-approx2}
Define
    \[f_e(s)=\int_{\partial\Omega_s} u_e^2\,\de\sigma_s.\]
We first show that $f_e(s)$ decays exponentially. Recall the parameter~$\delta_0$ defined at the beginning of~\S\ref{subsec:boundary-Delta-estimates}.
\begin{proposition}
\label{prop:f-exponential-decay}
Assume $\beta>\beta_0(\Omega)$ is large. The function $f_e$ is monotonically decreasing in $[0,\delta_0/4]$ and 
\[\forall s\in[0,\delta_0/4] \quad f_e(s)\leq e^{-s\sqrt{\beta}}f_e(0).\]
\end{proposition}
\begin{proof}
Note that by the first variation formula (see e.g. \cite[ch.\ 2]{simon1984lectures}, \cite[ch. 1]{li-geometric-analysis},\cite[ch.\ 4]{lin2002geometric}, \cite[ch.\ 1]{colding-minicozzi-minimal-surfaces})
    \begin{align}\label{eq:fe'}
        f_e'(s) &= \int_{\partial \Omega_s} \partial_s(u_e^2) + H_s u_e^2\,\de\sigma_s =-\int_{\Omega_s} \Delta(u_e^2)\,\de x +\int_{\partial\Omega_s} H_s u_e^2\,\de\sigma_s \\ \nonumber
        &=-2\int_{\Omega_s} (|\nabla u_e|^2 + \beta u_e^2)\,\de x +
        \int_{\partial\Omega_s} H_s u_e^2\,\de\sigma_s\leq C f_e(s),
    \end{align}
for some positive constant $C= C(\Omega)$. Applying again the first variation formula on the second term we get
    \begin{align*}
        f_e''(s) &= \int_{\partial\Omega_s} (2\beta+\partial_s H_s)u_e^2 + 2|\nabla u_e|^2 -(\partial_s u_e)^2
        + (H_s u_e+\partial_s u_e)^2\,\de\sigma_s \\
        &= \int_{\partial\Omega_s} (2\beta+\partial_s H_s)u_e^2 + |\nabla u_e|^2 +|\nabla_{g_s} u_e|^2
        + (H_s u_e+\partial_s u_e)^2\,\de\sigma_s \geq
        1.8\beta f_e(s)
    \end{align*}
    once
    \begin{align*}
          \beta \geq \beta_0:=5\sup_{x:\ \dist(x, \partial\Omega) < \delta_0} \bigl|(\partial_s H_s)(x)\bigr| .
    \end{align*}
    
Set $g(s)=e^{s\sqrt{1.8\beta}}f_e(s)$. By the preceding inequalities $g' \leq (C+\sqrt{1.8\beta})g$ and $g'' \ge (2\sqrt{1.8\beta}) g'$. From the latter inequality it follows that if $g'(s_0)>0$ then $g'(s)>0$ for all $s>s_0$. We claim that $g' \le 0.2\sqrt{\beta}g$ for all $s<\delta_0/4$. Otherwise, let $s_0<\delta_0/4$ be such that $g'(s_0) > (0.2\sqrt{\beta})g(s_0)$. Then for all $s \in (s_0,\delta_0)$,
\begin{align*}
    \Bigl(\frac{g'}{g}\Bigr)'(s) = \frac{gg''-g'^2}{g^2}(s) \ge
    \frac{(2\sqrt{1.8\beta})gg'-(C+\sqrt{1.8\beta})gg'}{g^2}(s)= \bigl(\sqrt{1.8\beta}-C\bigr)\frac{g'}{g}(s) .
\end{align*}
We conclude that
\begin{align*}
    \frac{g'}{g}(s) \ge \frac{g'}{g}(s_0)e^{(\sqrt{1.8\beta}-C)(s-s_0)} >
    0.2\sqrt{\beta}e^{(\sqrt{1.8\beta}-C)(s-\delta_0/4)} .
\end{align*}
Taking $s = \delta_0 / 2$, we get
\begin{align*}
    C+\sqrt{1.8\beta} \ge \frac{g'}{g}(\delta_0 / 2) > 0.2\sqrt{\beta}e^{(\sqrt{1.8\beta}-C)\delta_0/4} ,
\end{align*}
which cannot hold for large $\beta$. It follows that for all $s\in[0,\delta_0/4]$ we have $g'(s) \le 0.2\sqrt{\beta}g(s)$, and thus the inequality $f_e'(s) \le -\sqrt{\beta}f_e(s)$ holds, implying the stated decay.
\end{proof}

\begin{proof}[Proof of inequality~\eqref{ineq:main exp} in \Cref{thm:approx-and-decay}]
Combining the preceding proposition with \Cref{prop:boundary laplacian estimates}
we obtain
    \begin{align*}
        f_e(s) \leq e^{-s\sqrt{\beta}} f_e(0) = e^{-s\sqrt{\beta}} \int_{\partial \Omega} u_e^2\,\de\sigma \leq C e^{-s\sqrt{\beta}} \int_{\Omega} u^2\,\de x.
    \end{align*}
\end{proof}
\begin{proof}[Proof of inequality \eqref{ineq:main global} in \Cref{thm:approx-and-decay}] 
    We note that
    \begin{align} \label{eq:decomposition}
        \int_{\Omega_{s}} u_e ^2\,\de x = 
        \int_{\Omega_s\setminus\overline{\Omega}_{\delta_0/4}} u_e^2\,\de x + \int_{\Omega_{\delta_0/4}} u_e^2\,\de x.
    \end{align}
    The first term is estimated by integrating inequality~\eqref{ineq:main exp}:
    \begin{align} \label{eq:bound-u_e-omega-s-delta}
        \left\| u_{e}\right\|_{L^2(\Omega_{s}\setminus\overline{\Omega}_{\delta_0/4})}^{2}&=
        \int_{s}^{\delta_0/4}\int_{\partial\Omega_{t}}u_{e}^{2}\,\de\sigma_{t}\,\de t \le C\left\|u\right\|_{L^{2}(\Omega)}^{2}\int_{s}^{\delta_0/4}e^{-t\sqrt{\beta}}\,\de t
        \\
        &=\left\Vert u\right\Vert _{L^{2}(\Omega)}^{2}\frac{e^{-s\sqrt{\beta}}-e^{-\delta_0\sqrt{\beta}/4}}{\sqrt{\beta}} \le
        C\beta^{-1/2}e^{-s\sqrt{\beta}}\left\Vert u\right\Vert _{L^{2}(\Omega)}^{2} .\nonumber
    \end{align}
    To estimate the second term in~\eqref{eq:decomposition} we notice that $u_e^2$ is a subharmonic function. Therefore, according to Lemma~\ref{lem:subharmonic-estimate} below, for $s\in[0,\delta_0/8]$
    \begin{align}\label{eq:bound-norm-omega-delta}
        \int_{\Omega_{\delta_0/4}} u_e^2\,\,\de x \le
        C(\delta_0) \int_{\Omega_s \setminus \overline{\Omega}_{\delta_0/4}} u_e^2\,\de x .
    \end{align}
    Combining ~\eqref{eq:decomposition}, ~\eqref{eq:bound-u_e-omega-s-delta} and ~\eqref{eq:bound-norm-omega-delta} we get
    \begin{align*}
        \int_{\Omega_{s}} u_e^2\,\de x \le
        C\beta^{-1/2}e^{-s\sqrt{\beta}}\left\|u\right\|_{L^2(\Omega)}^{2}
    \end{align*}
    for some $C>0$ only depending on $\Omega$, proving the required inequality.
\end{proof}

\begin{lemma} \label{lem:subharmonic-estimate}
Let $v$ be a non-negative subharmonic function in a bounded domain $\Omega\subset\R^d$.
Let $\Omega_1\Subset \Omega$.
Then,
\[ \|v\|_{L^1(\Omega_1)}\leq \frac{ \mathrm{Vol}(\Omega_1) \cdot 2^d}{\omega_d\dist(\Omega_1, \partial\Omega)^{d}}\|v\|_{L^1(\Omega\setminus\overline{\Omega}_1)} ,\]
where $\omega_d$ is the volume of the unit ball.
\end{lemma}
\begin{proof}
Let $\Omega_1\Subset\Omega_2\Subset\Omega$. Let $x_0\in\partial\Omega_2$ be such that $v(x_0)=\max_{x\in \overline{\Omega}_2} v(x)$. For every $x\in\Omega_1$, we have $v(x)\leq v(x_0)$. Integrating over $x\in\Omega_1$ we obtain
\[\int_{\Omega_1} v(x)\,\de x \leq v(x_0)\mathrm{Vol}(\Omega_1).\]
Let $r>0$ be such that $B(x_0, r)\subset \Omega\setminus \overline{\Omega}_1$.
\[v(x_0)\leq \frac{1}{|B(x_0, r)|}\int_{B(x_0, r)} v(x)\,\de x\leq \frac{1}{|B_r|}\int_{\Omega\setminus\overline{\Omega}_1} v(x)\,\de x.\]
\end{proof}

\section{Proof of Theorem~\ref{thm:delocalization}: Delocalization near the boundary}
\label{sec:proof of main localization}

To prove \Cref{thm:delocalization}, for any $s \in [0,\delta_0]$, where~$\delta_0$ is as defined at the beginning of~\S\ref{subsec:boundary-Delta-estimates}, we define
        \begin{align}
            f_o(s) \coloneqq \int_{\partial \Omega_s} u_o^2\,\de\sigma_s ,
        \end{align}
    The following proposition shows a logarithmic convexity property for $f_o(s)$.

\begin{proposition}\label{prop:uo-log-cvx}
There is a positive constant $C= C(\Omega)$ such that
    \begin{align*}
       (f_o'(s))^2 \leq f_o(s) \left(f_o''(s)+C \beta \int_{\Omega}  u ^2\,\de x\right)
    \end{align*}
    for all $\beta\geq \beta_0(\Omega)$ large enough.
\end{proposition}
\begin{proof}
We calculate (cf.~\eqref{eq:fe'})
    \begin{align}\label{eq:derivative-fo}
         f_o'(s) = \int_{\partial \Omega_s} u_o \cdot (2\partial_s u_o +
         H_s u_o)\,\de\sigma_s
    \end{align}
and
    \begin{align}\label{eq:2nd derivative-fo}
        f_o''(s) &= \int_{\partial \Omega_s}  (2\partial_s u_o+H_s u_o)^2- 2(\partial_s u_o)^2 +2 u_o \partial^2_{s} u_o + (\partial_s H_s) u_o ^2\,\de\sigma_s .
    \end{align}
    By Cauchy-Schwarz inequality we have 
    \[(f_o')^2\leq f_o  \int_{\partial\Omega_s} (2\partial_s u_o+H_s u_o)^2\,\de\sigma_s. \]
Thus, by \eqref{eq:derivative-fo} and \eqref{eq:2nd derivative-fo}, after denoting
    \begin{align}\label{eq:def E}
        E(s) \coloneqq \int_{\partial \Omega_s} (\partial_s u_o)^2 - u_o \partial^2 _{s}u_o - \frac{\partial_s H_s}{2} u_o ^2\,\de\sigma_s,
    \end{align}
we see that
    \begin{align}\label{eq:fo inequality}
        f_o''-\frac{(f_o')^2}{f_o}\geq -2E(s).
    \end{align}
We claim that there is a constant $C(\Omega)>0$ such that when $s \in [0,\delta_0]$, 
        \begin{align}\label{eq:E bound}
            |E(s)| \leq  C \beta \int_{\Omega}  u^2\,\de x.
        \end{align}
\Cref{prop:uo-log-cvx} follows directly from \eqref{eq:fo inequality} and \eqref{eq:E bound}.

To prove \eqref{eq:E bound} we first show that
\begin{align} \label{eq:E0-bound}
    |E(0)| \leq C\beta \int_{\Omega}  u^2\,\de x .
\end{align}
Using the same operator $A(s) = \eta(s) \partial_s$ as in \Cref{prop:boundary laplacian estimates} and the fact that $(\Delta+\alpha) u_o = 0$, we have the following estimate for the first two terms in \eqref{eq:def E}:
\begin{align} \label{eq:bound-first-E0}
        \biggl|\int_{\partial \Omega} (\partial_s u_o)^2 &- u_o \partial^2 _{s}u_o\,\de\sigma\biggr| =
         \biggl|\int_{\partial \Omega} \partial_s u_o \cdot (Au_o) - u_o \cdot \partial_s (Au_o)\,\de\sigma  \biggr|
        \nonumber \\  &=  \biggl|-\int_\Omega  (\Delta+\alpha)u_o \cdot (Au_o) - u_o \cdot (\Delta+\alpha)(Au_o)\,\de x  \biggr|
        \nonumber \\  &=  \biggl|-\int_{\Omega} u_o \cdot [A,\Delta+\alpha]u_o\,\de x \biggr|
        \le \| u_o \|_{L^2} \cdot
        \bigl\| [A,\Delta+\alpha]u_o \bigr\| _{L^2}
        \\ & \le \| u \|_{L^2} \cdot
        C (\| \nabla u_o \|_{L^2} + \| \nabla ^2 u_o \|_{L^2})
        \nonumber \\ & \le \| u \|_{L^2} \cdot
        C (\| \nabla u \|_{L^2} + \| \nabla ^2 u \|_{L^2} + \frac{1}{\beta} \| \nabla ^3 u \|_{L^2} + \frac{1}{\beta} \| \nabla ^4 u \|_{L^2})
        \nonumber\\ & \le C \beta \| u \|^2_{L^2}
        ,\nonumber
\end{align}
where in the first inequality, we used the fact that $\bigl|[A,\Delta+\alpha]u_o\bigr| = \bigl|[A,\Delta]u_o\bigr| \leq C( |\nabla u_o| +|\nabla^2 u_o| ) $, and in the second inequality, we used the fact that $u_o = \frac{1}{\alpha+\beta}(\beta-\Delta) u $ together with \Cref{lem:w22} and \Cref{lem:w42}. For the third term in \eqref{eq:def E}, since $u=u_e+u_o=0$ on $\partial \Omega$, we have by \Cref{prop:boundary laplacian estimates}
\begin{align}\label{eq:bound-second-E0}
    \int_{\partial \Omega} |\partial_s H_0| u_o^2\,\de\sigma =
    \int_{\partial \Omega} |\partial_s H_0| u_e^2\,\de\sigma \leq
    C \int_{\Omega} u^2 \,\de x.
\end{align}
The estimate \eqref{eq:E0-bound} follows from ~\eqref{eq:bound-second-E0} and ~\eqref{eq:bound-first-E0}.

Next, we estimate $| E'(s)|$. We rewrite $E(s)$ as follows (see~\eqref{eqn:Laplace in Fermi}):
\begin{align*}
    E(s) & = \int_{\partial \Omega_s} (\partial_s u_o)^2 + H_s u_o\partial_s u_o + u_o\Delta_{g_s} u_o + \alpha u_o ^2 - \frac{\partial_s H_s}{2} u_o^2\,\de\sigma_s
    \\ & = \int_{\partial \Omega_s} (\partial_s u_o)^2 + H_s u_o \partial_s u_o - | \nabla_{g_s} u_o |^2 + \alpha u_o ^2 - \frac{\partial_s H_s}{2} u_o^2\,\de\sigma_s.
\end{align*}
We differentiate $E(s)$ using the first variation formula to get:
\begin{align}\label{eq:E derivative}
        E'(s) &= \int_{\partial\Omega_{s}}
        2(\partial_{s}u_{o})(\partial_{s}^{2}u_{o} + H_s \partial_s u_o + \alpha u_o) +
        H_{s}u_{o}(\partial_{s}^{2}u_{o} + H_s \partial_s u_o + \alpha u_o)
        \nonumber \\
        &-2 \langle \nabla_{g_{s}}u_{o},\nabla^{\mathbb{E}}_{\partial_{s}}\nabla_{g_{s}}u_{o} \rangle 
        -\frac{\partial_{s}^{2}H_{s}}{2}u_{o}^{2} 
        -H_{s}\left( |\nabla_{g_{s}}u_{o}|^{2} +
        \frac{\partial_{s}H_{s}}{2}u_{o}^{2}\right)\,\de\sigma_{s}
        \nonumber \\
        & =\int_{\partial\Omega_{s}}
        -2(\partial_{s}u_{o})\Delta_{g_s}u_o
        -u_{o}H_{s}\Delta_{g_s}u_o
        \nonumber \\
        &-2 \langle \nabla_{g_{s}}u_{o},\nabla^{\mathbb{E}}_{\partial_{s}}\nabla_{g_{s}}u_{o} \rangle 
        -\frac{\partial_{s}^{2}H_{s}}{2}u_{o}^{2} 
        -H_{s}\left( |\nabla_{g_{s}}u_{o}|^{2} +
        \frac{\partial_{s}H_{s}}{2}u_{o}^{2}\right)\,\de\sigma_{s}
        \\
        & =\int_{\partial\Omega_{s}}
        2 \langle [\nabla_{g_s}, \nabla^{\mathbb{E}}_{\partial_{s}}]u_{o},\nabla_{g_s}u_o \rangle
        + \langle \nabla_{g_s}(u_{o}H_{s}),\nabla_{g_s}u_o \rangle
        \nonumber \\
        &-\frac{\partial_{s}^{2}H_{s}}{2}u_{o}^{2} 
        -H_{s}\left( |\nabla_{g_{s}}u_{o}|^{2} +
        \frac{\partial_{s}H_{s}}{2}u_{o}^{2}\right)\,\de\sigma_{s}
        \nonumber \\
         & \stackrel{(*)}{=}\int_{\partial\Omega_{s}}
        2 \text{II}_s(\nabla_{g_s}u_o,\nabla_{g_s}u_o)
        + \langle \nabla_{g_s}(u_{o}H_{s}),\nabla_{g_s}u_o \rangle
        \nonumber \\
        &-\frac{\partial_{s}^{2}H_{s}}{2}u_{o}^{2} 
        -H_{s}\left( |\nabla_{g_{s}}u_{o}|^{2} +
        \frac{\partial_{s}H_{s}}{2}u_{o}^{2}\right)\,\de\sigma_{s} .
        \nonumber
\end{align}
Thus, 
    \begin{align*}
        |E'(s)| \leq C (\Omega)\int_{\partial \Omega_s} u_o ^2 + |\nabla u_o|^2\,\de\sigma_s.
    \end{align*}
In~$(*)$ we have used the following identity for any function $v$ and vector field $Y$ tangent to $\partial\Omega_s$.
\begin{align*}
    &\langle \nabla_{g_s} \partial_s v, Y\rangle -\langle \nabla^\bE_{\partial_s} \nabla_{g_s} v, Y\rangle= 
    Y\partial_s v - \partial_s\langle  \nabla_{g_s} v, Y\rangle+\langle \nabla_{g_s} v, \nabla^{\bE}_{\partial_s}Y\rangle\\
    &=[Y,\partial_s] v+
    \langle \nabla_{g_s} v, \nabla^{\bE}_Y \partial_s\rangle+\langle \nabla_{g_s} v, [\partial_s, Y]\rangle=\langle \nabla_{g_s} v, \nabla^{\bE}_Y \partial_s \rangle=\text{II}_s(\nabla_{g_s} v, Y).
\end{align*}
Then, integrating \eqref{eq:E derivative}, taking into account \eqref{eq:E0-bound}, we have
\begin{align*}
    |E(s)| &
    \leq |E(0)| + C \int_{\Omega} u_o ^2 + |\nabla u_o |^2 \,\de x
    \\
    &\leq C \beta \int_{\Omega} u^2 \,\de x + C \int_{\Omega} u^2 \,\de x + C\int_{\Omega} |\nabla u|^2 \,\de x+ \frac{1}{\beta^2} \int_{\Omega} |\nabla^3 u|^2 \,\de x
    \leq C \beta \int_{\Omega} u^2 \,\de x,
\end{align*}
where in the second inequality we used the fact that $u_o = \frac{1}{\alpha+\beta}(\beta-\Delta) u $, and in the last inequality, we used \Cref{lem:w22} and \Cref{lem:w42}.
\end{proof}

\Cref{thm:delocalization} follows from integrating the next delocalization estimate.
\begin{proposition}\label{prop:fo(s) bound}
There are positive constants $C= C(\Omega)$ and $\beta_0 = \beta_0(\Omega)$, such that if $\beta > \beta_0$, then when $s \in [0,\delta_0/2]$,
    \begin{align*}
        f_o(s)\leq C(1+\beta s^2) \int_{\Omega}  u^2 \,\de x .
    \end{align*}
\end{proposition}
\begin{proof}
     To simplify expressions, assume  $\int_{\Omega}  u^2 \,\de x = 1$. When $s \in [0,\delta_0/2]$, if $f_o(s)=0$ then there is nothing to prove. If $f_o(s)>0$ then we claim that for the constant $C = C(\Omega)$ in \Cref{prop:uo-log-cvx}, 
        \begin{align}\label{eq:sqrt fo derivative bound}
            \left(\sqrt{f_o}\right)'(s) \leq \sqrt{C \beta}.
        \end{align}
    Otherwise, assume that there is an $s_0 \in [0, \delta_0/2]$, such that $f_o(s_0)>0$ and $\left(\sqrt{f_o}\right)'(s_0) > \sqrt{C \beta}$, or equivalently $f_o'(s_0) > 2 \sqrt{C\beta f_o(s_0)}$. We see that
    \begin{align*}
        \left(\sqrt{f_o}\right)''(s_0) &=
        \frac{1}{4 f_o(s_0) ^{3/2}}  \left( 2f_o''(s_0) f_o(s_0) - f_o'(s_0)^2\right)  
        \\
        \stackrel{\text{Prop. } \ref{prop:uo-log-cvx}}{\geq} &
        \frac{1}{4 f_o(s_0) ^{3/2}}  \left(f_o'(s_0)^2 - 2C\beta f_o(s_0)\right)
        \geq \frac{2C\beta f_o(s_0)}{4f_o(s_0) ^{3/2}} >0.
    \end{align*}
    Thus, a continuity argument implies that $f_o(s)>0$ and $\left(\sqrt{f_o}\right)'(s) > \sqrt{C \beta}$ for all $s\geq s_0$. Since $s_0 \in [0, \delta_0/2]$ we conclude that for any $s \in [\delta_0/2,\delta_0]$,
        \begin{align*}
            \sqrt{f_o(s)} \geq \sqrt{f_o(s)} - \sqrt{f_o(\delta_0/2)} =\int_{\delta_0/2} ^s \left(\sqrt{f_o}\right)'(t) \,\de t \geq \sqrt{C\beta}(s-\delta_0/2).
        \end{align*}
    Thus, 
    \begin{align}
        1 = \int_{\Omega}  u^2 \,\de x \geq \int_{ \Omega_{\delta_0/2}} u_o ^2 (x) \,\de x \geq \int_{\delta_0/2} ^{\delta_0} f_o(s)\, \de s \geq C\beta \int_{\delta_0/2} ^{\delta_0} (s-\delta_0/2)^2\, \de s = \frac{C \beta \delta_0 ^3}{24},
    \end{align}
which is a contradiction if $\beta$ is large. Hence, we obtain \eqref{eq:sqrt fo derivative bound}.

Then, for any $s\in[0,\delta_0/2]$ such that $f_o(s) >0$, if $f_o$ has no roots in $[0,s]$,
\begin{align*}
    \sqrt{f_o(s)} \overset{\eqref{eq:sqrt fo derivative bound}}{\leq} \sqrt{f_o(0)} + \sqrt{C\beta}s \overset{\text{Prop. } \ref{prop:boundary laplacian estimates}}{\leq} C(1+\sqrt{\beta}s),
\end{align*}
Otherwise, let $s_0$ be the root of $f_o$ closest to $s$ such that $s_0 < s$. Then
\begin{align*}
    \sqrt{f_o(s)} = \sqrt{f_o(s)} - \sqrt{f_o(s_0)} =
    \int_{s_0}^{s} \Bigl( \sqrt{f_o(t)} \Bigr)'\, \de t
    \overset{\eqref{eq:sqrt fo derivative bound}}{\leq} C\sqrt{\beta}s
\end{align*}
as required.
\end{proof}

\section{Proof of Theorem~\ref{thm:whispering gallery}: Boundary localization in balls}\label{sec:sharp localization rate}

In this section, we study the optimal boundary localization on the unit balls $\Omega = B_1 \subseteq \bR^d$ ($d \geq 2$). We aim to prove the following \Cref{thm:sharp localization on disk} which explicitly gives the function $\gamma(\eps)$ in \Cref{thm:whispering gallery}.
\begin{theorem}\label{thm:sharp localization on disk}
    Let $\Omega = B_1 \subseteq \bR^d$ ($d\geq 2$). There exists a sequence of $L^2$-normalized eigenfunctions $(u_{m})_{m=1} ^{\infty}$ with eigenvalues $(\lambda_m)_{m=1} ^{\infty}$, such that for any $a >0$,
    \begin{align}
        \varliminf_{m \to \infty} \int_{\Omega \setminus \Omega_{a \lambda_m ^{-1/6}}} u_{m} ^2 \,\de x \geq (1 - 400e^{-\frac{1}{4}a^{3/2}}).
    \end{align}
    Thus, $(u_{m})_{m=1} ^{\infty}$ concentrates near the boundary at the scale $\lambda_m ^{-1/6}$ as $m \to \infty$.
\end{theorem}
To prove \Cref{thm:sharp localization on disk}, we need some basic facts on Bessel functions and we collect the following explicit descriptions of solutions to \eqref{eq:clamped-plate} in $\Omega = B_1 \subseteq \bR^d$.

For $\mu \geq 0$, let $J_\mu(s), I_\mu(s)$  denote the Bessel function and modified Bessel function of the first kind of order~$\mu$ respectively. 
For $d\geq 2$ set $D=d/2-1$. The ($d$-dimensional) spherical Bessel function and modified spherical Bessel function are  defined respectively by
    \begin{align*}
        \calJ_m (s) \coloneqq s^{-D} J_{m+D}(s) , \quad \calI_m(s)\coloneqq s^{-D} I_{m+D}(s) , \quad m \in \bN_{0}.
    \end{align*}
The Wronskian of $\calJ_m(s)$ and $\calI_m(s)$ is defined by
    \begin{align}
        \calW_m \coloneqq \calJ_m\calI_m' - \calJ_m'\calI_m=
     \calJ_{m} \calI_{m+1}+\calJ_{m+1}\calI_{m}
    \end{align}
    Observe that $\calW_m(s)=s^{-2D}W_{\mu}(s)$ with $\mu =  m+D$, where $W_\mu$ is the Wronskian of~$J_\mu$ and~$I_\mu$. In particular, $\calW_m$ and $W_{\mu}$ have the same roots.
We let $w_{\mu,1}$ denote the first positive root of~$W_\mu$. We consider a subsequence of solutions to Problem~\eqref{eq:clamped-plate} of the form  (see e.g.~\cite[ch. V\S6]{CH89})
    \begin{align}\label{e:sharp eigenfunction}
        u_m(x) \coloneqq \left[ \frac{\calJ_m(w_{\mu, 1} r)}{\calJ_m(w_{\mu, 1} )}  -  \frac{\calI_m(w_{\mu, 1} r)}{\calI_m(w_{\mu, 1} )} \right] Y_m (\theta),\quad m\in \bN_{0},
    \end{align}
where $x=r\theta$ with $\theta\in \SS^{d-1}$ and  $Y_m$ is any $L^2$-normalized  spherical harmonic of degree~$m$. The function~$u_m$ is an eigenfunction of eigenvalue $\lambda=w_{\mu, 1}^4$.
Note that
 $u_{m, o} = r^{-D}\frac{J_\mu(w_{\mu, 1} r)}{J_\mu(w_{\mu, 1} )}Y_m(\theta)$ and  $u_{m,e} = -r^{-D}\frac{I_\mu(w_{\mu, 1} r)}{I_\mu(w_{\mu, 1} )}Y_m(\theta)$.
The main proposition is the boundary localization of the oscillating part $u_{m, o}$.
\begin{proposition}
\label{prop:umo-localisation}
Fix any $a \geq 4$. For all~$m$ large enough depending on~$a$, one has that
     \begin{align*}
            \int_{B_{1-a\lambda^{-1/6}}} u_{m, o}^2\,\de x \leq 100 e^{-\frac{1}{2}a^{3/2}} \cdot \int_{B_1} u_{m, o}^2\,\de x.
        \end{align*}
\end{proposition}
The exponential part does not contribute to localization, since it is exponentially small.
\begin{proposition}
\label{prop:ume-decay}
    Fix any $a\geq 1$. For all~$m$ large enough depending on $a$, one has that
    \begin{align*}
    \int_{B_{1-a\lambda^{-1/6}}} u_{m, e}^2\,\de x \leq e^{-\frac{1}{4}a\lambda^{1/12}}\int_{B_1} u_{m}^2\,\de x 
    \end{align*}
    where $\lambda=w_{\mu, 1}^4$ is the eigenvalue corresponding to~$u_m$.
\end{proposition}
\begin{proof}[Proof of~\Cref{prop:ume-decay}]
This proposition follows from the inequality~\eqref{ineq:main global} by setting $s = a \lambda^{-\frac{1}{6}}$.
\end{proof}

We first note that~\Cref{thm:sharp localization on disk} follows from~\Cref{prop:umo-localisation} and~\Cref{prop:ume-decay}. 
\begin{proof}[Proof of~\Cref{thm:sharp localization on disk}]
Write $s = 1-a\lambda^{-1/6}$.
From Propositions~\ref{prop:umo-localisation} and~\ref{prop:ume-decay}, we see that for all~$m$ large enough and satisfying $\lambda^{\frac{1}{12}} > a^{1/2}$, we have that
\begin{align*}
\int_{B_s} u_m^2\,\de x &\leq 2\int_{B_s} u_{m, o}^2\,\de x +  2 \int_{B_{s}} u_{m, e}^2\,\de x \\ 
&\leq 200 e^{-\frac{1}{4}a^{3/2}}\Bigl(\int_{B_1} u_{m, o}^2\,\de x + \int_{B_1} u_{m}^2\,\de x\Bigr)\leq 400 e^{-\frac{1}{4}a^{3/2}}\int_{B_1} u_{m}^2\,\de x\ .
\end{align*}
\end{proof}
To prove Proposition~\ref{prop:umo-localisation}, we need a few lemmas regarding $L^2$-integrals and asymptotics of Bessel functions at the transition regime.

\begin{lemma}\label{lem:u_mo-square-integral}
    For any $m \in \bN_0$ and $s \geq 0$,
        \begin{align*}
            \int_{B_s} u_{m, o}^2(x)\,\de x = \frac{\bigl((w_{\mu, 1} s)^2 - \mu^2 \bigr)  J_{\mu}(w_{\mu, 1} s)^2 + (w_{\mu, 1} s)^2 J_{\mu}'(w_{\mu, 1} s)^2}{2J_\mu(w_{\mu, 1}) ^{2} \cdot w_{\mu, 1} ^2 } . 
        \end{align*}
\end{lemma}
\begin{proof}
This is equivalent to the following classical Lommel's integral (\cite[\S5.11]{W1995} or \cite[Lemma 4.1]{lin2026whispering}).
        \begin{align*}
            \int_0^s J_{\mu}^2 (t) t\,\de t =\frac{(s^2 - \mu^2)}{2}  J_{\mu}(s)^2 + \frac{s^2}{2}J_\mu'(s)^2.
        \end{align*}
\end{proof}

\begin{lemma}[Zeros comparison and estimates]
\label{lem:roots-asymptotics}
The following properties hold
\begin{enumerate}[label=
\textup{(\alph*)}]
\item\label{itm:jm-interlacing} {\textup{\cite[\S15.22]{W1995}}}
    For all $\mu \geq 0$, 
    \[\mu < j_{\mu, 1} <j_{\mu+1, 1}< j_{\mu,2}\] 
\item \label{itm:j-asymp} {\textup{\cite[\S10.21.40]{NIST-handbook}}}    For all $\mu$ large enough,
        \begin{align*}
            \mu+ 1.855 \cdot \mu^{\frac{1}{3}}< j_{\mu, 1} < \mu + 1.856 \cdot \mu^{\frac{1}{3}}.
        \end{align*}
\item  \label{itm:j-w-zeros}   For all $\mu\geq 0$, $j_{\mu, 1} <w_{\mu, 1} <j_{\mu+1, 1}$, 

\item \label{itm:w-asymp} For all $\mu$ large enough we have
        \begin{align*}
            \mu + 1.855 \cdot \mu^{\frac{1}{3}}< w_{\mu, 1} <  \mu + 2 \cdot \mu^{\frac{1}{3}}.
        \end{align*}
\end{enumerate}
    \end{lemma}

\begin{proof}
    We prove Parts~\ref{itm:j-w-zeros} and~\ref{itm:w-asymp}. The inequality $j_{\mu, 1}<w_{\mu, 1}$ follows from Part~\ref{itm:jm-interlacing} and the formula 
    \[ W_{\mu}=J_{\mu} I_{\mu+1}+ I_\mu J_{\mu+1}.\]
To see that $w_{\mu, 1}<j_{\mu+1, 1}$ use the same formula and the fact that $j_{\mu+1, 1}<j_{\mu, 2}$ to conclude that $W_\mu(j_{\mu+1, 1})<0$. 
Part~\ref{itm:w-asymp} follows from Parts~\ref{itm:j-asymp},~\ref{itm:j-w-zeros} and 
the inequality $1+1.856(\mu+1)^{1/3}<2\mu^{1/3}$ which holds for all $\mu>500$.
\end{proof}

\begin{remark}\label{rem:limit expansion}
    It could be that one has the stronger asymptotic expansion  $w_{\mu, 1} = \mu + C \mu^{\frac{1}{3}} + o(\mu^{1/3})$, but we do not use it.
\end{remark}

We also note that 
if we set $r_\mu=(\mu-(a-2)\mu^{1/3})/w_{\mu, 1}$ we have
\begin{lemma} 
For all $\mu$ large enough
    \[r_\mu > 1- a \lambda^{-1/6}\]
\end{lemma}
\begin{proof}
From \Cref{lem:roots-asymptotics}\ref{itm:w-asymp} we see 
    \begin{align*}
        r_\mu =1-\frac{w_{\mu,1}-\mu+(a-2)\mu^{1/3}}{w_{\mu,1}} >1-\frac{a\mu^{1/3}}{w_{\mu,1}} >1-a w_{\mu,1}^{-2/3}=1-a \lambda^{-\frac{1}{6}}.
    \end{align*}
\end{proof}

When considering the asymptotics of $J_\mu(w_{\mu, 1}r_\mu)$, the Airy function comes into play. We summarize the properties we need in the following lemma. 
\begin{lemma}
\label{lem:transition-asymptotics}
    For any fixed $a\geq 0$, as $\mu\to\infty$,
    \begin{enumerate}[label=\textup{(\alph*)}]
\item  \textup{\cite[\S10.19(iii)]{NIST-handbook}} $J_\mu(\mu-a\mu^{\frac{1}{3}}) = 2^\frac{1}{3} \mu^{-\frac{1}{3}} \Airy(2^{\frac{1}{3}}a) + O_a(\mu^{-1})$,
\item \textup{\cite[\S10.19(iii)]{NIST-handbook}} 
$J_\mu'(\mu-a\mu^{\frac{1}{3}}) = -2^\frac{2}{3} \mu^{-\frac{2}{3}} \Airy' (2^{\frac{1}{3}}a) + O_a(\mu^{-\frac{4}{3}})$,
\item \label{itm:airy 1} \textup{\cite[\S9.7(iii)]{NIST-handbook}} $\forall s>1$, $|\Airy'(s)|<s^{\frac{1}{4}}e^{-\frac{2}{3}s^{3/2}}$,
\item \label{itm:airy 2}\textup{\cite[\S9.2(ii)]{NIST-handbook}} $\Airy'(0)<-0.25$ .
\end{enumerate}        
\end{lemma}

We can now prove Proposition~\ref{prop:umo-localisation}.
\begin{proof}[Proof of~\Cref{{prop:umo-localisation}}]
     We show that for all $m$ large enough one has
     \begin{align*}
            \int_{B_{r_\mu}} u_{m, o}^2\,\de x \leq 100 e^{-\frac{1}{2}a^{3/2}} \cdot \int_{B_1} u_{m, o}^2\,\de x
        \end{align*}
        Denote by $T_\mu(s)$ the numerator in the expression
        of the right hand side in~\Cref{lem:u_mo-square-integral}. 
     Since $w_{\mu, 1}r_\mu <\mu$ and $w_{\mu, 1}^2r_{\mu}^2<\mu^2+a^2\mu^{2/3}$, we obtain by~\Cref{lem:transition-asymptotics}
        \begin{align*}
            T_{\mu}(r_\mu)
            \leq a^2\mu^{\frac{2}{3}} J_\mu(w_{\mu, 1}r_\mu)^2 + \mu^{2} J_\mu'(w_{\mu, 1}r_\mu)^2  =\mu^{\frac{2}{3}} 2^{\frac{4}{3}} \left(\Airy' (2^{\frac{1}{3}}(a-2)) \right)^2 + O_a(1).
        \end{align*}
On the other hand 
\begin{equation*}
    T_{\mu}(1)> T_{\mu}(\mu/w_{\mu, 1}) = \mu^2 J_{\mu}'(\mu)^2 = \mu^{\frac{2}{3}} 2^{\frac{4}{3}} \Airy'(0)^2 +O_a(1)
\end{equation*}
By \ref{itm:airy 1} and \ref{itm:airy 2} in \Cref{lem:transition-asymptotics}, we get from the preceding two inequalities that for all large~$\mu$
\[T_\mu(r_\mu)<20 (a-2)^{1/2} e^{-\frac{4\sqrt{2}}{3}(a-2)^{3/2}} T_\mu(1).\]
When $a \geq 4$, we have that $20(a-2)^{1/2}e^{-\frac{4\sqrt{2}}{3}(a-2)^{3/2}} \leq 100 e^{-\frac{1}{2}a^{3/2}}$. 
This finishes the proof for~\Cref{{prop:umo-localisation}}.
\end{proof}


\section{Proof of Theorem \ref{thm:delocalization on cylinder}: Delocalization in a Finite Flat Cylinder}
\label{sec:proof cylinder}

In this section, we show that there is no localization neither near the boundary nor inside the domain on the flat cylinder $\Sigma = \bS^1\times[-h, h]$. \Cref{thm:strong delocalization on cylinder2} below  explicitly gives the function $\delta(s)$ in \Cref{thm:delocalization on cylinder}. In accordance with  our notation~\eqref{eq:omegas} we let $\Sigma_{s h} = \bS^1\times[-(1-s)h, (1-s)h]$.

\begin{theorem}\label{thm:strong delocalization on cylinder2}
    Let $\Sigma = \bS^1\times[-h, h]$ be a cylinder of height $2h$ and radius $1$. Let $\{(u_{\lambda},\lambda)\}_{\lambda}$ be any family of normalized solution pairs to \eqref{eq:clamped-plate}. Then for any $s\in (0,1)$,
    \begin{align}
    0<\frac{1}{\pi}\bigl(\pi s - \sin(\pi s)\bigr)\leq \varliminf_{\lambda\to \infty} \|u_\lambda\|^2_{L^2(\Sigma_{(1-s)h})}\leq \varlimsup_{\lambda\to \infty} \|u_\lambda\|^2_{L^2(\Sigma_{(1-s)h})}\leq  \frac{1}{\pi}\bigl(\pi s + \sin(\pi s)\bigr)<1 .
    \end{align}
\end{theorem}

We first describe a basis of clamped plate eigenfunctions on the cylinder in the following \Cref{prop:cylinder eigenfunction form2}.
Let $f_{m,c},f_{m,s}:(m,\infty)\to \bR$ be defined by
\begin{align*}
    &f_{m,c}(x) := h(x^2-m^2)^{1/2}\tan\bigl( h(x^2-m^2)^{1/2} \bigr) +
    h(x^2+m^2)^{1/2}\tanh\bigl(h(x^2+m^2)^{1/2}\bigr)
    \\
    &f_{m,s}(x) := h(x^2-m^2)^{1/2}\cot\bigl(h(x^2-m^2)^{1/2}\bigr) -
    h(x^2+m^2)^{1/2}\coth\bigl(h(x^2+m^2)^{1/2}\bigr)
\end{align*} 
Denote by $\alpha_{m,k}$ and $\beta_{m,k}$ the $k$-th positive roots of $f_{m, c}$ and $f_{m, s}$,
respectively, and set
\begin{align*}
    a^{-}_{m, k}=(\alpha_{m,k}^2-m^2)^{1/2}, \quad\quad a^{+}_{m, k}=(\alpha_{m, k}^2+m^2)^{1/2}
    \\
    b^{-}_{m, k}=(\beta_{m,k}^2-m^2)^{1/2}, \quad\quad b^{+}_{m, k}=(\beta_{m, k}^2+m^2)^{1/2}.
\end{align*}
\begin{proposition}\label{prop:cylinder eigenfunction form2}
    For any $m \in \bN_0$ and $k \in \bN$, define
    \begin{align*}
        T_{m,k,c}(t) :=
        \frac{\cos(t a^{-}_{m,k})}{\cos(h a^{-}_{m,k})} - 
        \frac{\cosh(t a^{+}_{m,k})}{\cosh(h a^{+}_{m,k})}
        \quad and \quad
        T_{m,k,s}(t) :=
        \frac{\sin(t b^{-}_{m,k})}{\sin(h b^{-}_{m,k})} - 
        \frac{\sinh(t b^{+}_{m,k})}{\sinh(h b^{+}_{m,k})},
    \end{align*}
    for $t \in [-h,h]$. 
    Then,
        \begin{align}\label{eq:torus basis}
            \{ T_{m,k,s}\cos m\theta,\ T_{m, k, s}\sin m\theta,\ T_{m,k,c}\cos m\theta, \ T_{m, k, c}\sin m\theta\}_{m\in \bN_0, k\in \bN}
        \end{align}
        is an orthogonal basis of~$L^2(\Sigma)$ composed of clamped-plate eigenfunctions. The eigenvalue corresponding to $T_{m,k,c}\cos m\theta$ and $T_{m,k,c}\sin m\theta$ is $\alpha_{m,k}^4$ and the one corresponding to $T_{m, k, s}\cos m\theta$ and $T_{m ,k, s}\sin m\theta$ is $\beta_{m,k}^4$.
\end{proposition}
\begin{proof}
This is a standard separation of variables argument.
\end{proof}
The following lemma will be used in the proof of Theorem \ref{thm:strong delocalization on cylinder2}. Its proof can be found at the end of this section.
\begin{lemma}\label{lem:alpha beta asymptotics2}
    Let $m \in \bN_0$. For any $k \in \bN$, we have:
    \begin{align}
        &\pi k - \pi/2< ha^{-}_{m,k} < \pi k \label{eq:range alpha2},
        \\
        &\underset{m\to \infty}{\lim} ha^{-}_{m,k} = \pi k - \pi/2 \label{eq:lim alpha2},
        \\
        &\pi k < hb^{-}_{m,k} < \pi k + \pi/4 \label{eq:range beta2},
        \\
        &\underset{m\to \infty}{\lim} hb^{-}_{m,k} = \pi k .\label{eq:lim beta2}
    \end{align}
In particular, for all $m$ and $k$ we have $\alpha_{m, k}\neq\beta_{m, k}$, and any two distinct basis eigenfunctions in~\eqref{eq:torus basis} with the same eigenvalue are orthogonal on $\Sigma_{sh}$ for any $s \in (0,1)$.
\end{lemma}
\subsection{Proof of Theorem \ref{thm:strong delocalization on cylinder2}}
Let $u_\lambda$ be a normalized solution to \eqref{eq:clamped-plate}. Then $u_\lambda$ is a linear combination of basis eigenfunctions in \eqref{eq:torus basis} all of  the same eigenvalue. By \Cref{lem:alpha beta asymptotics2}, these basis eigenfunctions are orthogonal on $\Sigma_{sh}$. Thus, it suffices to prove \Cref{thm:strong delocalization on cylinder2} when $u_\lambda$ is one of the basis eigenfunctions. In the following we consider the case $u_\lambda(t, \theta) = T_{m,k,c}(t)\cos m\theta$. The proofs of the other cases are similar.
Write
\begin{align} \label{eq:cylinder-norm}
    \frac{\norm{u_\lambda}^2_{L^2(\Sigma_{h(1-s)})}}
    {\norm{u_{\lambda}}^2_{L^2(\Sigma)}} =
    \frac{\int_{-sh}^{sh} \left(\cos(ta^{-}_{m,k}) - 
    \cos(ha^{-}_{m,k})\frac{\cosh(ta^{+}_{m,k})}{\cosh(ha^{+}_{m,k})} \right)^2 \,\de t}
    {\int_{-h}^{h} \left(\cos(ta^{-}_{m,k}) - 
    \cos(ha^{-}_{m,k})\frac{\cosh(ta^{+}_{m,k})}{\cosh(ha^{+}_{m,k})} \right)^2 \,\de t}
\end{align}
First observe that
\begin{align} \label{eq:lim-numerator}
    \varliminf_{\lambda \to \infty} & \int_{-sh}^{sh} \left(\cos(t a^{-}_{m,k}) - 
    \cos(h a^{-}_{m,k})\frac{\cosh(t a^{+}_{m,k})}{\cosh(h a^{+}_{m,k})} \right)^2 \,\de t =
    \\
    & \varliminf_{\lambda\to\infty} \int_{-sh}^{sh} \cos^2(ta_{m, k}^{-})\,\de t =
    sh + \varliminf_{\lambda \to \infty} \frac{\sin(2sh a^{-}_{m,k})}{2a_{m,k}^{-}} \nonumber
\end{align}
since $\frac{\cosh(ta_{m, k}^{+})}{\cosh{(ha_{m,k}^{+})}}\to 0$ (as $a_{m,k}^{+} \to \infty$) when $t \in (-h,h)$. Note:
\begin{enumerate}[label=(\roman*)]
    \item If $k \to \infty $ then by (\ref{eq:range alpha2}) $h a^{-}_{m,k} \to \infty$ and so $\frac{\sin(2sh a^{-}_{m,k})}{2a^{-}_{m,k}} \to 0$.
    \item \label{itm:partial-lim} For fixed~$k$ and $m \to \infty$, by (\ref{eq:lim alpha2}) $h a^{-}_{m,k} \to \pi k - \frac{\pi}{2}$. Hence, we get
    \begin{align*}
        \lim_{m \to \infty} \Big| \frac{\sin(2sh a^{-}_{m,k})}{2 a^{-}_{m,k}} \Big| =
        \Big| h \frac{\sin((2k-1)\pi s)}{(2k-1)\pi} \Big| \leq
        \Big|h \frac{\sin(\pi s)}{\pi} \Big|
    \end{align*}
    where we used the inequality:
    \begin{align*}
        \forall n\in\bN , \ \forall x\in\bR: \quad
        \Bigl|\frac{\sin (nx)}{nx}\Bigr|\leq\Bigl|\frac{\sin x}{x}\Bigr|.
    \end{align*}
\end{enumerate}
Also observe that the above reasoning with $s=1$ shows that the denominator in \eqref{eq:cylinder-norm} tends to $h$.
Hence, from \eqref{eq:lim-numerator} and \ref{itm:partial-lim} we conclude that
\begin{align*}
    \varliminf_{\lambda \to \infty} \|  u_{\lambda}\|^2_{L^2(\Sigma_{h(1-s)})} \geq
    \frac{1}{\pi}(\pi s - \sin(\pi s)) 
\end{align*}
The argument for the upper limit is essentially the same.
\begin{proof}[Proof of Lemma~\ref{lem:alpha beta asymptotics2}]
    Fix $m$. To prove (\ref{eq:range alpha2}), let $x^-=(x^2 - m^2)^{1/2}$ and $x^+=(x^2+m^2)^{1/2}$. At $hx^-=\pi k$ we have
    \begin{align*}
        f_{m,c}(x) &= hx^-\tan{hx^-} + hx^+\tanh{hx^+} >0
    \end{align*}
    On the other hand  $f_{m, c}(x)\to -\infty$ as $hx^{-}\searrow \pi k-\pi/2$. Since $f_{m, c}$ is monotonically increasing in the interval $(\pi k-\pi/2, \pi k+\pi/2)$ inequality~(\ref{eq:range alpha2}) follows.
    
    To prove (\ref{eq:lim alpha2}) write
    \begin{align*}
        ha_{m,k}^-\tan(h a_{m,k}^-)+h a_{m, k}^+\tanh h a_{m, k}^+ = 0
    \end{align*}
    As $m\to\infty$ we have $a_{m, k}^+\to\infty$ and so the second term tends to infinity. Hence, the first term tends to~$-\infty$.
    Since $h a_{m, k}^-$ lies in $(\pi k-\pi/2, \pi k+\pi/2)$ we conclude
    that $h a_{m,k}^- \to \pi k - \frac{\pi}{2}$.

    The proofs of (\ref{eq:range beta2}) and (\ref{eq:lim beta2}) are similar.

    Finally, the orthogonality of two different basis eigenfunctions in \eqref{eq:torus basis} with the same eigenvalue follows directly from \eqref{eq:range alpha2}, \eqref{eq:range beta2}, and their orthogonality on $\bS^1$.
\end{proof}







\bibliographystyle{abbrv} 
\bibliography{Bi_Laplacian.bib}

\end{document}